\documentclass[a4paper,12pt]{amsart}
\usepackage{ifthen}
\provideboolean{submission}
\setboolean{submission}{false}

\usepackage{amsmath,amssymb,amsthm}
\ifthenelse{\boolean{submission}}
{}
{\usepackage[hmarginratio={1:1},vmarginratio={1:1},heightrounded,textwidth=400pt]{geometry}}

\newtheorem{theorem}{Theorem}

\newtheorem{corollary}[theorem]{Corollary}

\newtheorem{lemma}[theorem]{Lemma}

\let\leq\leqslant
\let\geq\geqslant
\let\setminus\smallsetminus
\let\epsi\varepsilon

\newcommand{\brac}[1]{{\left(#1\right)}}
\newcommand{\set}[1]{\left\{#1\right\}}
\newcommand{\norm}[1]{{\left|#1\right|}}

\newcommand{\ceil}[1]{{\left\lceil #1 \right\rceil}}

\newcommand{\Oh}[1]{O\brac{#1}}
\newcommand{\oh}[1]{o\brac{#1}}
\newcommand{\Om}[1]{\Omega\brac{#1}}

\newcommand{\Ot}[1]{\Theta\brac{#1}}

\newcommand{\Nat}{\mathbb{N}}

\newcommand{\dir}[1]{\vec{#1}}

\newcommand{\kG}{\mathbb{G}}
\newcommand{\kH}{\mathbb{H}}

\makeatletter
\newcommand\etc{etc\@ifnextchar.{}{.\@}}
\newcommand\etal{et~al\@ifnextchar.{}{.\@}}
\makeatother

\begin{document}

\ifthenelse{\boolean{submission}}
{\begin{frontmatter}}
{}

\ifthenelse{\boolean{submission}}
{\title{Universal targets for homomorphisms\break of edge-colored graphs}}
{\title{Universal targets for homomorphisms\break of edge-colored graphs}
\thanks{Research of the first author was supported by MNiSW grant DI2013 000443.
Research of the second author was supported by NCN grant UMO-2011/03/D/ST6/01370.}}

\ifthenelse{\boolean{submission}}
{\author{Grzegorz Gu\'{s}piel\fnref{t1}}
\ead{guspiel@tcs.uj.edu.pl}
\fntext[t1]{Partially supported by MNiSW grant DI2013 000443.}}
{\author[G.~Gu\'{s}piel]{Grzegorz Gu\'{s}piel}}

\ifthenelse{\boolean{submission}}
{\author{Grzegorz Gutowski\fnref{t2}}
\ead{gutowski@tcs.uj.edu.pl}
\fntext[t2]{Partially supported by NCN grant UMO-2011/03/D/ST6/01370.}}
{\author[G.~Gutowski]{Grzegorz Gutowski}}

\address{Theoretical Computer Science Department, Faculty of Mathematics and Computer Science, Jagiellonian University, Krak\'{o}w, Poland}
\ifthenelse{\boolean{submission}}
{}
{\email{\{guspiel,gutowski\}@tcs.uj.edu.pl}}

\begin{abstract}

A $k$-edge-colored graph is a finite, simple graph with edges labeled by numbers $1,\ldots,k$.
A function from the vertex set of one $k$\nobreakdash-edge-colored graph to another is a homomorphism if the endpoints of any edge are mapped to two different vertices connected by an edge of the same color.
Given a class $\mathcal{F}$ of graphs, a~$k$-edge-colored graph $\kH$ (not necessarily with the underlying graph in~$\mathcal{F}$) is $k$\nobreakdash-universal for $\mathcal{F}$ when any $k$-edge-colored graph with the underlying graph in $\mathcal{F}$ admits a homomorphism to $\kH$.
We characterize graph classes that admit $k$\nobreakdash-universal graphs.
For such classes, we establish asymptotically almost tight bounds on the size of the smallest universal graph.

For a nonempty graph $G$, the density of $G$ is the maximum ratio of the number of edges to the number of vertices ranging over all nonempty subgraphs of $G$.
For a nonempty class $\mathcal{F}$ of graphs, $D(\mathcal{F})$ denotes the density of $\mathcal{F}$, that is the supremum of densities of graphs in $\mathcal{F}$.

The main results are the following.
The class $\mathcal{F}$ admits $k$\nobreakdash-universal graphs for $k\geq2$ if and only if there is an absolute constant that bounds the acyclic chromatic number of any graph in $\mathcal{F}$.
For any such class, there exists a constant $c$, such that for any $k \geq 2$, the size of the smallest $k$\nobreakdash-universal graph is between $k^{D(\mathcal{F})}$ and $ck^\ceil{D(\mathcal{F})}$.

A connection between the acyclic coloring and the existence of universal graphs was first observed by Alon and Marshall ({\em Journal of Algebraic Combinatorics}, 8(1):5--13, 1998).
One of their results is that for the class of planar graphs, the size of the smallest $k$\nobreakdash-universal graph is between $k^3+3$ and $5k^4$.
Our results yield that there exists a constant~$c$ such that for all $k$, this size is bounded from above by $ck^3$.

\end{abstract}

\ifthenelse{\boolean{submission}}
{\end{frontmatter}}
{\maketitle}

\section{Introduction}

All graphs considered in this paper are finite and contain no loops or~multiple edges.
For simplicity, by a class of graphs we mean a nonempty set of graphs.
The set of natural numbers $\set{0, 1, \ldots}$ is denoted by $\Nat$.
For every $k \in \mathbb{N}$, the set $\{1,\ldots,k\}$ is denoted by $[k]$.
For two real-valued functions $f$ and $g$ whose domains are cofinite subsets of $\Nat$, we write $f(k) = \Oh{g(k)}$ if there exist constants $n_0$ and $c$ such that $\norm{f(k)} \leq c\norm{g(k)}$ for all $k \geq n_0$;
         $f(k) = \Om{g(k)}$ if $g(k) = \Oh{f(k)}$;
     and $f(k) = \Ot{g(k)}$ if both $f(k) = \Oh{g(k)}$ and $f(k) = \Om{g(k)}$ hold.
We write $f(k) = \oh{g(k)}$ if for every $\epsi > 0$ there exists a constant $n_0$ such that $\norm{f(k)} \leq \epsi\norm{g(k)}$ for all $k \geq n_0$.

A \emph{$k$-edge-colored graph} $\kG$ is a pair $(G,c)$, where $G$ is a graph, called \emph{an~underlying graph} of $\kG$, and $c$ is a mapping from $E(G)$ to $[k]$, called a~\emph{$k$\nobreakdash-edge-coloring} of $\kG$.
A \emph{$k$-edge-colored graph over} $G$ is a $k$-edge-colored graph with the underlying graph $G$.

Let $\kG_1 = (G_1,c_1)$ and $\kG_2=(G_2,c_2)$ be two $k$-edge-colored graphs.
A~mapping $h:V(G_1) \to V(G_2)$ is a \emph{homomorphism} of $\kG_1$ to $\kG_2$ if, for every two vertices $u$ and $v$ that are adjacent in $G_1$, $h(u)$ and $h(v)$ are adjacent in $G_2$ and $c_1(\{u,v\}) = c_2(\{h(u),h(v)\})$.
In other words, a homomorphism of $\kG_1$ to $\kG_2$ maps every colored edge in $\kG_1$ into an edge of the same color in $\kG_2$.

A $k$-edge-colored graph $\kH$ is \emph{$k$-universal} for a class $\mathcal{F}$ of graphs if every $k$\nobreakdash-edge-colored graph over any graph in $\mathcal{F}$ admits a homomorphism to $\kH$.
We denote by $\lambda_{\mathcal{F}}(k)$ the minimum possible number of vertices in a~$k$\nobreakdash-universal graph for $\mathcal{F}$.
We set $\lambda_{\mathcal{F}}(k) = \infty$ if such a graph does not exist.
The main result of this paper is a characterization of graph classes that admit $k$\nobreakdash-universal graphs.
For any such class $\mathcal{F}$ of graphs, the asymptotic behavior of $\lambda_{\mathcal{F}}(k)$ is determined.

Observe that $\lambda_{\mathcal{F}}(1)$ is the maximum chromatic number of all graphs in $\mathcal{F}$.
Although this parameter is of great importance in graph theory, this paper is focused on the behavior of $\lambda_{\mathcal{F}}(k)$ when $k$ tends to infinity.
In particular, the case $k=1$ differs significantly from the case $k \geq 2$.
Only the latter one is the subject of this paper.

The crucial notion that helps to determine if a given graph class admits $k$\nobreakdash-universal graphs is the acyclic coloring.
An \emph{acyclic coloring} of a graph is an assignment of colors to the vertices of the graph such that:
\begin{itemize}
\item[(i)] every two adjacent vertices get different colors,
\item[(ii)] vertices of any cycle in the graph get at least $3$ different colors.
\end{itemize}
In other words, an acyclic coloring is a proper coloring such that, for any two colors, the graph induced by vertices of these two colors is a forest.
The \emph{acyclic chromatic number} of a graph $G$, denoted $\chi_a(G)$, is the minimum number of colors in an acyclic coloring of $G$.
We give the following characterization of graph classes that admit $k$\nobreakdash-universal graphs.
\begin{theorem}
\label{thm:main_theorem_existence}
Let $k \geq 2$.
A class $\mathcal{F}$ of graphs admits a $k$\nobreakdash-universal graph if and only if there is an absolute constant $r$ such that $\chi_a(G) \leq r$ for every $G$ in $\mathcal{F}$.
\end{theorem}
In particular, this theorem gives that a graph class either admits a $k$\nobreakdash-universal graph for all $k \geq 2$ or for no $k \geq 2$.

A strong connection between the acyclic coloring and the existence of universal graphs was first noted by Alon and Marshall~\cite{AlonM98}.
They proved that if the acyclic chromatic number of any graph in $\mathcal{F}$ is at~most $r$, then $\mathcal{F}$ admits a $k$\nobreakdash-universal graph on at~most $rk^{r-1}$ vertices.
This shows that the bounded acyclic number is a sufficient condition for a class $\mathcal{F}$ of graphs to admit $k$\nobreakdash-universal graphs.
We show that this condition is also necessary.

Alon and Marshall~\cite{AlonM98} used their result to construct a small $k$\nobreakdash-universal graph for the class of planar graphs $\mathcal{P}$.
Their technique, combined with the famous result of Borodin~\cite{Borodin79} that every planar graph has acyclic chromatic number at most $5$, gives $\lambda_{\mathcal{P}}(k) = \Oh{k^4}$.
Alon and Marshall gave a lower bound $\lambda_{\mathcal{P}}(k) = \Om{k^3}$ and asked for the exact asymptotics of $\lambda_{\mathcal{P}}(k)$.

Theorem~\ref{thm:main_theorem_upper_bound} and Theorem~\ref{thm:main_theorem_lower_bound_density} allow to determine the asymptotics of $\lambda_{\mathcal{F}}(k)$ for any class $\mathcal{F}$ of graphs of bounded acyclic chromatic number.
In particular, for planar graphs we obtain that $\lambda_{\mathcal{P}}(k) = \Ot{k^3}$.
In general, we show that the asymptotic behavior of $\lambda_{\mathcal{F}}(k)$ for a class $\mathcal{F}$ of graphs of bounded acyclic chromatic number depends on the density of $\mathcal{F}$.
The \emph{density} of a~graph $G$, denoted $D(G)$, is defined as 
 $$D(G) = \max \set{ \frac{\norm{E(G')}}{\norm{V(G')}}\ :\ G' \text{ is a nonempty subgraph of } G }\mathrm{,}$$
and the \emph{density} of a class $\mathcal{F}$ of graphs, denoted $D(\mathcal{F})$, is given by
 $$D(\mathcal{F}) = \sup \set{ D(G)\ :\ G \in \mathcal{F} }\mathrm{.}$$

Hakimi~\cite{Hakimi65} was the first to observe that graphs of low density admit orientations of low in-degree.
An \emph{orientation} of a graph $G$ is an assignment of direction to each edge of $G$, which turns $G$ into an oriented graph~$\dir{G}$.
If an edge $\{a,b\}$ of $G$ is oriented from $a$ to $b$, then $(a,b)$ is an edge in $\dir{G}$, $a$ is the \emph{tail}, and $b$ is the \emph{head} of $(a,b)$.
\emph{In-degree} of a vertex $b$ is the number of different edges in $\dir{G}$ with head $b$.
An orientation $\dir{G}$ of $G$ is a \emph{$d$-orientation} if every vertex of $\dir{G}$ has in-degree at most $d$.

By pigeonhole principle, a graph $G$ does not admit a $\brac{\ceil{D(G)}-1}$-ori\-en\-ta\-tion.
On the other hand, Hakimi~\cite{Hakimi65} proved that any graph $G$ admits a $\ceil{D(G)}$-orientation.

\begin{theorem}[Hakimi~\cite{Hakimi65}]
\label{thm:Hakimi}
Let $\mathcal{F}$ be a class of graphs.
Every graph in $\mathcal{F}$ admits a $d$-orientation if and only if the density of $\mathcal{F}$ satisfies $D(\mathcal{F}) \leq d$.
\end{theorem}

The next lemma shows a connection between the acyclic chromatic number and orientations of low in-degree.

\begin{lemma}
\label{lem:bounded_acyclic_orientation}
If a graph $G$ admits an acyclic coloring with $r$ colors, then $G$ admits an $(r-1)$-orientation.
\end{lemma}
\begin{proof}
Suppose that there is an acyclic coloring of $G$ with colors in the set $[r]$.
We show that $G$ admits an $(r-1)$-orientation.
Fix $i,j \in [r]$, $i < j$.
Since the graph induced by the vertices colored $i$ or $j$ is a forest, there is an~orientation of this forest such that every vertex is a head of at most one edge.
If we repeat this argument for all pairs of colors, we obtain an $(r-1)$-orientation of $G$.
\end{proof}
Using Theorem~\ref{thm:Hakimi} and Lemma~\ref{lem:bounded_acyclic_orientation} we immediately obtain the following lemma.
\begin{lemma}
\label{lem:bounded_acyclic_density}
Let $\mathcal{F}$ be a class of graphs for which there is an absolute constant~$r$ such that $\chi_a(G) \leq r$ for every $G$ in $\mathcal{F}$.
The density $D(\mathcal{F})$ is bounded and
 $$D(\mathcal{F}) \leq r-1\mathrm{.}$$
\end{lemma}

Observe that there are graph classes of bounded density and unbounded acyclic chromatic number.
For an example, let $K^\star_n$ be a graph obtained from the clique $K_n$ by subdividing every edge exactly once.
One may easily verify that the graph class $\{K^\star_n: n \in \mathbb{N}\}$ has density $2$ and unbounded acyclic chromatic number.

The main result of this paper is the following theorem.

\begin{theorem}
\label{thm:main_theorem_upper_bound}
Let $\mathcal{F}$ be a class of graphs for which there are absolute constants $r$ and $d$ such that every graph in $\mathcal{F}$ admits both an acyclic coloring with $r$ colors and a $d$-orientation.
For any $k \geq 2$, the following holds:
 $$\lambda_{\mathcal{F}}(k) \leq 8dr^4 \binom{8dr^4}{d} k^{d}\mathrm{.}$$
In particular, $\lambda_{\mathcal{F}}(k) = \Oh{k^{\ceil{D(\mathcal{F})}}}$.
\end{theorem}
Observe that for any class $\mathcal{F}$ of graphs of acyclic chromatic number at most $r$, Lemma~\ref{lem:bounded_acyclic_density} guarantees that $D(\mathcal{F}) \leq r - 1$.
Thus, the bound $\lambda_{\mathcal{F}}(k) = \Oh{k^{\ceil{D(\mathcal{F})}}}$ given by Theorem~\ref{thm:main_theorem_upper_bound} is asymptotically no worse than the bound $\lambda_{\mathcal{F}}(k)= \Oh{k^{r-1}}$ obtained by Alon and Marshall~\cite{AlonM98}.
In Section~\ref{sec:applications} we present some natural graph classes for which our bound is significantly better.

The following results show that the upper bound of Theorem~\ref{thm:main_theorem_upper_bound} is asymptotically almost tight.

\begin{lemma}
\label{lem:lower_bound_density}
Let $G$ be a graph, let $k \geq 2$, and let $\kH$ be a $k$-edge-colored graph such that any $k$-edge-colored graph over $G$ admits a homomorphism to $\kH$.
The $k$-edge-colored graph $\kH$ has at least $k^{D(G)}$ vertices.
\end{lemma}
\begin{proof}
Let $G'$ be a nonempty subgraph of $G$ and $c_{G'}$ be some $k$-edge-coloring of $G'$.
By the assumption of the theorem, the $k$-edge-colored graph~$\kG' = (G', c_{G'})$ admits a homomorphism to $\kH$.
Since different colorings need different homomorphisms, we have that the number of $k$-edge-colorings of~$G'$ is at most the number of different functions from $V(G')$ to $V(\kH)$,
 $$k^{\norm{E(G')}} \leq \norm{V(\kH)}^{\norm{V(G')}}\mathrm{.}$$
It follows that $k^{\frac{\norm{E(G')}}{\norm{V(G')}}} \leq \norm{V(\kH)}$ holds for any nonempty subgraph $G'$ of~$G$ and
$$k^{ D(G) } \leq \norm{V(\kH)}\mathrm{.}$$
\end{proof}
Lemma~\ref{lem:lower_bound_density} yields the following lower bound for $\lambda_{\mathcal{F}}(k)$.

\begin{theorem}
\label{thm:main_theorem_lower_bound_density}
For any class $\mathcal{F}$ of graphs, the following holds:
 $$\lambda_{\mathcal{F}}(k) \geq k^{D(\mathcal{F})}\mathrm{.}$$
\end{theorem}
\begin{proof}
Fix $k \geq 2$.
Suppose $\kH$ is a $k$\nobreakdash-universal graph for $\mathcal{F}$.
For every $\epsilon > 0$ there is a graph $G$ in $\mathcal{F}$ such that $D(G) \geq D(\mathcal{F}) - \epsilon$.
Using Lemma~\ref{lem:lower_bound_density} we get
 $$\norm{V(\kH)} \geq k^{D(G)} \geq k^{D(\mathcal{F}) - \epsilon}.$$
Since the inequality above holds for every $\epsilon > 0$, the claim of the theorem follows.
\end{proof}

To sum up, for a class $\mathcal{F}$ of graphs of bounded acyclic chromatic number we have
$$\lambda_{\mathcal{F}}(k) = \Om{k^{D(\mathcal{F})}} \text{ and } \lambda_{\mathcal{F}}(k) = \Oh{k^{ \ceil{D(\mathcal{F})} }} \mathrm{,}$$
which is tight for graph classes of integral density.
For other graph classes, we suspect that the lower bound describes the correct asymptotics of $\lambda$.
The multiplicative constant hidden by the asymptotic notation depends on the bound on the acyclic chromatic number and is quite big.
For small values of $k$, the upper bound by Alon and Marshall is substantially better.

\section{Applications.}\label{sec:applications}

In this section we give two examples of application of our results.
We establish the asymptotics of $\lambda$ for the class $\mathcal{P}$ of planar graphs and for the class $\mathcal{G}_g$ of graphs embeddable on an oriented surface of genus $g$.

\subsection{Planar graphs.}
Borodin~\cite{Borodin79} showed that the acyclic chromatic number of any planar graph is at most 5.
The density $D(\mathcal{P})$ of planar graphs is~$3$.
Consequently, Theorem~\ref{thm:main_theorem_upper_bound} and Theorem~\ref{thm:main_theorem_lower_bound_density} yield the following.
\begin{corollary}
 $$\lambda_{\mathcal{P}}(k) = \Ot{k^3}\mathrm{.}$$
\end{corollary}
This answers the question asked by Alon and Marshall~\cite{AlonM98}.

\subsection{Graphs of bounded genus.}
Let $\mathcal{G}_g$ denote the class of all graphs embeddable on an oriented surface of genus $g$.
Alon, Mohar, and Sanders~\cite{AlonMS94} showed that the acyclic chromatic number of any graph in $\mathcal{G}_g$ is at most $\Oh{g^{\frac{4}{7}}}$.
They also provide a construction of a graph in $\mathcal{G}_g$ of acyclic chromatic number at least $\Om{g^{\frac{4}{7}}/(\log{g})^{\frac{1}{7}}}$.
We show that the density $D(\mathcal{G}_g)$ is of a smaller order:
  $$\sqrt{3g} - \frac{1}{2} \leq D(\mathcal{G}_g) \leq \sqrt{3g}+3\mathrm{.}$$

For the lower bound, let $t = \ceil{ \sqrt{12g} }$, and let $K_{t}$ be a clique on $t$ vertices.
Thanks to Ringel and Youngs~\cite{RingelY68} and other authors, we know that $K_t$ is embeddable on an oriented surface of genus $\ceil{\frac{(t-3)(t-4)}{12}}$.
Thus, $K_{t}$ is in $\mathcal{G}_g$.
The density of $K_{t}$ equals
 $$D(K_{t}) = \frac{\norm{E(K_{t})}}{\norm{V(K_{t})}} = \frac{t \cdot (t-1)}{2 \cdot t} = \frac{\ceil{\sqrt{12g}} -1}{2} \geq \sqrt{3g}-\frac{1}{2}\mathrm{,}$$
which proves the lower bound.

For the upper bound, let $G$ be a graph in $\mathcal{G}_g$ and $G'$ be a nonempty subgraph of $G$ with the highest ratio $\frac{\norm{E(G')}}{\norm{V(G')}}$.
If $\norm{V(G')} \leq t$, then $D(G) \leq \sqrt{3g}$ as the density of a graph on $t$ vertices does not exceed the density of $K_t$.
Let $n > t$, $m$, and $f$ denote respectively the number of vertices, edges, and faces in some embedding of $G'$ on an oriented surface of genus $g$.
By Euler's formula, we have $n-m+f = 2-2g$.
Multiplying this equality by $3$ and plugging $3f \leq 2m$ we get that $m \leq 3n - 6 + 6g$.
Since $n > t$ and $t = \ceil{ \sqrt{12g} }$ we obtain
 $$\frac{m}{n} \leq 3 - \frac{6}{n} + \frac{6g}{n} \leq \sqrt{3g}+3\mathrm{,}$$
which proves the upper bound.

Consequently, Theorem~\ref{thm:main_theorem_upper_bound} and Theorem~\ref{thm:main_theorem_lower_bound_density} yield the following.
\begin{corollary}
  $$
  k^{\sqrt{3g}-\frac{1}{2}} \leq
  \lambda_{\mathcal{G}_g}(k) = \Oh{k^{\ceil{\sqrt{3g}+3}}}
  \mathrm{.}$$
\end{corollary}

\section{Main results}

The proofs of Theorem~\ref{thm:main_theorem_existence} and Theorem~\ref{thm:main_theorem_upper_bound} use the notion of star coloring.
A \emph{star coloring} of a graph is an~assignment of colors to the vertices of the graph such that:
\begin{itemize}
\item[(i)] every two adjacent vertices get different colors,
\item[(ii)] every subsequent four vertices on any path in the graph get at least $3$ different colors.
\end{itemize}
In other words, a star coloring is a proper coloring such that, for any two colors, every connected component in the graph induced by vertices of these two colors has at most one vertex of degree higher than one.
The \emph{star chromatic number} of a graph $G$, denoted $\chi_s(G)$, is the minimum number of colors in a star coloring of $G$.
In particular, any star coloring of $G$ is an acyclic coloring of $G$.
On the other hand, Albertson \etal~\cite{AlbertsonCKKR04} showed that any acyclic coloring with $r$ colors can be used to construct a star coloring with at most $2r^2 - r$ colors.

We need to introduce yet another version of coloring that we use in our proofs.
Let $\dir{G}$ be an orientation of a graph $G$.
We use the following notions: if $(u,v)$ is an edge of $\dir{G}$, then $u$ is a \emph{parent} of $v$; if $(u,v)$ and $(v,w)$ are edges of $\dir{G}$, then $u$ is a \emph{grandparent} of $w$.
An \emph{out-coloring} of an oriented graph is an assignment of colors to the vertices of the graph such that:
\begin{itemize}
  \item[(C1)] every two adjacent vertices get different colors,
  \item[(C2)] every two distinct parents of a single vertex get different colors,
  \item[(C3)] any vertex and any of its grandparents get different colors.
\end{itemize}
Clearly, any out-coloring of $\dir{G}$ is a star coloring of $G$ and hence an~acyclic coloring of $G$.

The same coloring was studied before in the context of star colorings
by Ne\v{s}et\v{r}il and Ossona de~Mendez~\cite{NesetrilO03},
and as in-coloring by Albertson \etal~\cite{AlbertsonCKKR04}.
An \emph{in-coloring} of $\dir{G}$ is a proper coloring in which for every 2-colored path of 3 vertices in $\dir{G}$, the edges are directed towards the middle vertex of the path.
It is easy to see that a coloring of vertices of $\dir{G}$ is an out-coloring of $\dir{G}$ if and only if it is an in-coloring of the transpose of $\dir{G}$.

The following result allows a construction of an out-coloring with a~small number of colors for a graph with low star chromatic number and an orientation of low in-degree.
The proof of the lemma is similar to the proof of Theorem~2.1 in the article by Ne\v{s}et\v{r}il and Ossona de~Mendez~\cite{NesetrilO03}.

\begin{lemma}
\label{lem:star_coloring_implies_compatible_coloring}
If $\dir{G}$ is a $d$-orientation of $G$, then $\dir{G}$ admits an out-coloring with $2d \cdot \chi_s(G)^2$ colors.
\end{lemma}
\begin{proof}
Let $c_s$ be a star coloring of $G$ with $\chi_s(G)$ colors.
First, we define an auxiliary directed graph $H$ on the vertex set of $G$.
The edges of $H$ will encode conditions (C2) and (C3) for an out-coloring of $\dir{G}$ and are defined according to the following two rules:
\begin{itemize}
\item[(R1)] For every triple $(b, x, a)$ of vertices of $G$, when $a$ and $b$ are different parents of $x$ in $\dir{G}$ and $c_s(a) = c_s(b)$, we add an edge $(b, a)$ to $H$.
\item[(R2)] For every triple $(b, x, a)$ of vertices of $G$, when $b$ is a parent of $x$ in~$\dir{G}$, and $x$ is a parent of $a$ in $\dir{G}$, and $c_s(a) = c_s(b)$, we add an edge $(b, a)$ to $H$.
\end{itemize}
Observe that a single edge of $H$ may be added multiple times and that both edges $(a,b)$ and $(b,a)$ may be present in $H$ for some vertices $a$ and $b$.

To give an upper bound for the in-degree of a vertex $a$ in $H$, let $T_a$ be the set of all triples $(y,x,a)$ that add an edge in $H$ according to rules (R1) or (R2).
Observe that for any color $\alpha$ there is at most one $x$ with $c_s(x) = \alpha$ such that a triple $(y,x,a)$ is in $T_a$.
Suppose to the contrary that there are two triples $(y_1, x_1, a)$ and $(y_2, x_2, a)$ in $T_a$ with $c_s(x_1) = c_s(x_2)$ and $x_1 \neq x_2$.
Rules (R1) and (R2) ensure that $c_s(a) = c_s(y_1)$ and, as a consequence, the path $y_1, x_1, a, x_2$ in $G$ gets only two colors in $c_s$, a contradiction.
Furthermore, for any $x$ there are at most $d$ different triples $(y, x, a)$ in $T_a$, as $y$ needs to be a parent of $x$.
For $\alpha = c_s(a)$ there is no triple $(y,x,a)$ in $T_a$ with $c_s(x) = \alpha$, as $x$ and $a$ are neighbors in $G$.
Thus, the size of $T_a$, and effectively the in-degree of $a$, is at most $d(\chi_s(G)-1)$.

A simple counting argument shows that any induced subgraph of $H$ contains a vertex of degree at most $2d(\chi_s(G)-1) < 2d\chi_s(G)$.
This allows a~construction of a proper coloring of $H$ with $2d\chi_s(G)$ colors using the following strategy:
 pick a vertex $v$ of the smallest degree in $H$;
 recursively color the subgraph $H\setminus v$; 
 color $v$ with any color not assigned to the neighbors of $v$.
There are fewer than $2d\chi_s(G)$ neighbors, so there is such a color.
Let $c_H$ be the constructed coloring.

We define the coloring $c$ of $\dir{G}$ to be $c(v) = (c_s(v), c_H(v))$.
Clearly, $c$ uses at most $2d \cdot \chi_s(G)^2$ colors.
We claim that $c$ is an out-coloring of $\dir{G}$.
The condition (C1) is satisfied as $c_s$ is a proper coloring of $G$.
The construction of $H$ and $c_H$ ensure that conditions (C2) and (C3) are also satisfied.
\end{proof}

We present an auxiliary lemma that plays a crucial role in the proof of Theorem~\ref{thm:main_theorem_upper_bound}.
\begin{lemma}
\label{lem:main_theorem_upper_bound}
Let $\mathcal{F}$ be a class of graphs for which there are absolute constants $q$ and $d$ such that every graph in $\mathcal{F}$ admits a $d$\nobreakdash-orientation that has an out-coloring with $q$ colors.
For any $k \geq 2$, the following holds:
 $$\lambda_{\mathcal{F}}(k) \leq q\binom{q}{d}k^d\mathrm{.}$$
\end{lemma}
\begin{proof}
We explicitly construct a $k$\nobreakdash-universal graph $\kH=(H,c_H)$ for the class $\mathcal{F}$.
The vertex set of $\kH$ is the set of all $(q+1)$-tuples of the form
$$(i, x_1, x_2, \ldots, x_q),$$
where $i \in [q]$, and $x_j \in [k]$ for all $j \in [q]$, and where among $x_1, x_2, \ldots, x_q$ there are at most $d$ values different from $k$.
$\kH$ is a complete graph, i.e., there is an edge between any two vertices.
The $k$-edge-coloring of $\kH$ is given by:
$$c_H(\ \{ (i, x_1, x_2, \ldots, x_q),(j, y_1, y_2, \ldots, y_q) \}\ ) = \min(y_i,x_j).$$
This completes the definition of $\kH$.
Note that the vertex set of $\kH$ has size smaller than $q\binom{q}{d}k^d$.

Let $\kG = (G, c_G)$ be a $k$-edge-colored graph such that $G$ admits a $d$-orientation $\dir{G}$ that has an out-coloring $f$ with $q$ colors.
We define a~homomorphism $h$ of $\kG$ to $\kH$ given by:
\begin{displaymath}
h(u) = (f(u), x_1, x_2, \ldots, x_q),
\end{displaymath}
where for each $i \in [q]$
\begin{displaymath}
x_i =
\begin{cases}
  c_G(\set{u, p}) & \mbox{if } u \mbox{ has a parent } p \mbox{ in } \dir{G} \mbox{ with } f(p) = i \mbox{,} \\
  k & \mbox{otherwise.}
\end{cases}
\end{displaymath}
Thanks to condition (C2) for out-colorings, $u$ has at most one parent colored $i$.
$\dir{G}$ is a $d$-orientation and $u$ has at most $d$ parents.
Thus, $h(u)$ is properly defined for any vertex $u$, and $h$ maps $\kG$ to $\kH$.

To prove that $h$ is a homomorphism, consider two adjacent vertices $u, v$ in $\kG$.
Without loss of generality we may assume that $v$ is a parent of $u$ in~$\dir{G}$.
We have
\begin{eqnarray*}
h(u) &=& (f(u), x_1, x_2, \ldots, x_r), \\
h(v) &=& (f(v), y_1, y_2, \ldots, y_r).
\end{eqnarray*}
Clearly, by condition (C1), we have that $h(u) \neq h(v)$.
What remains is to prove that the color of the edge $\{u,v\}$ in $\kG$ is the same as the color of the edge $\set{h(u), h(v)}$ in $\kH$.
As $c_H(\set{h(u), h(v)}) = \min(x_{f(v)}, y_{f(u)})$ and $x_{f(v)} = c_G(\set{u, v})$, it remains to show that $y_{f(u)} = k$.
Observe that $y_{f(u)} \neq k$ implies that $v$ has a parent $w$ in $\dir{G}$ and $w$ is colored $f(u)$.
But then, $w$ is a grandparent of $u$ and $f(w)=f(u)$, which contradicts condition (C3) for out-colorings.
\end{proof}

\begin{proof}[Proof of Theorem~\ref{thm:main_theorem_upper_bound}]
Let $\mathcal{F}$ be a class of graphs such that each graph in $\mathcal{F}$ admits both an acyclic coloring with $r$ colors and a $d$-orientation.
Using the result of Albertson \etal~\cite{AlbertsonCKKR04} we get that every graph in $\mathcal{F}$ admits a~star coloring with $2r^2$ colors.
Using Lemma~\ref{lem:star_coloring_implies_compatible_coloring} we get that every graph in $\mathcal{F}$ admits a $d$-orientation that has an out-coloring with $8dr^4$ colors.
Then, using Lemma~\ref{lem:main_theorem_upper_bound} we get that $\mathcal{F}$ admits a $k$\nobreakdash-universal graph on
 $$8dr^4\binom{8dr^4}{d}k^d$$
vertices, which completes the proof.
\end{proof}

Before we prove Theorem~\ref{thm:main_theorem_existence}, we show the following technical lemma.
\begin{lemma}
\label{lem:lower_bounds_orientation_compatible}
Let $\mathcal{F}$ be a class of graphs such that $\mathcal{F}$ admits a $k$\nobreakdash-universal graph on $p$ vertices for some $k \geq 2$.
The following holds:
\begin{enumerate}
\item[(S1)] Every graph in $\mathcal{F}$ admits a $\ceil{ \log_k{p} }$-orientation.
\item[(S2)] For any $d \geq 1$, any $d$-orientation $\dir{G}$ of any graph $G \in \mathcal{F}$, $\dir{G}$ admits an out-coloring with $(2d+1)p^{\ceil{\log_k{d}}}$ colors.
\end{enumerate}
\end{lemma}
\begin{proof}
Statement (S1) follows directly from Lemma~\ref{lem:lower_bound_density} and Theorem~\ref{thm:Hakimi}.

For the proof of (S2), let $\dir{G}$ be a $d$-orientation of some graph $G$ in $\mathcal{F}$.
We explicitly construct an out-coloring of $\dir{G}$ with $(2d+1)p^{\ceil{\log_k{d}}}$ colors.
The first step is a construction of a coloring $c$ of $\dir{G}$ that satisfies conditions (C1) and (C2) for out-colorings.
Coloring $c$ uses at most $p^{\ceil{ \log_k{d} }}$ colors.

Let $m = \ceil{ \log_k{d} }$.
Any $x \in \mathbb{N}$ smaller than $k^m$ can be expressed in base~$k$ positional notation as follows: $x = \sum_{i=1}^m a_i(x)k^{i-1}$ where $a_i(x) \in \set{0, 1, \ldots, k - 1}$ for each $i \in [m]$.
Let $\set{f_1, f_2, \ldots, f_m}$ be a family of integer functions ($f_i: [d] \to [k]$) such that $f_i(x) = a_i(x-1) + 1$.
Observe that for any $a, b \in [d]$, $a \neq b$, there is an $i \in [m]$ such that $f_i(a) \neq f_i(b)$.

Each vertex $v$ has at most $d$ parents in $\dir{G}$.
Assign to every parent of~$v$ a~different number in the set $[d]$.
Let $p_i(v)$ denote the parent of $v$ with number~$i$.

For every $i \in [m]$, consider a $k$-edge-coloring $c_{i}$ of $G$ given by
$$c_i(\{v, p_j(v)\}) = f_i(j)\mathrm{.}$$
Using the fact that $\kH$ is $k$\nobreakdash-universal for $\mathcal{F}$, for every $i \in [m]$ let $h_i$ be a~homomorphism of $(G, c_i)$ to $\kH$.
We define coloring $c:V(G) \to V(\kH)^m$ as follows:
\begin{displaymath}
c(v) = (h_1(v), h_2(v), \ldots, h_m(v)).
\end{displaymath}

To see that $c$ satisfies (C1) note that any homomorphism maps adjacent vertices in $G$ to different vertices in $\kH$.

To prove that $c$ satisfies (C2) consider two different parents $p_a(v)$, $p_b(v)$ of a vertex $v$.
Let $i \in [m]$ be such that $f_i(a) \neq f_i(b)$.
By the definition of~$c_i$ we get that $c_i(\{v, p_a(v)\}) \neq c_i(\{v, p_b(v)\})$.
It follows that $h_i$ maps $p_a(v)$~and~$p_b(v)$ to different vertices of $\kH$ and coloring $c$ satisfies (C2).

Suppose that $(u,v)$ and $(v,w)$ are two edges in $\dir{G}$ and that $u = p_a(v)$ and $v = p_b(w)$.
If $a \neq b$ then coloring $c$ assigns different colors to vertices $u$~and~$w$ by the same argument as in the previous paragraph.
We need to refine coloring~$c$ so that condition (C3) is satisfied for pairs $u$, $w$ such that $u$ is a grandparent of $w$ and $u = p_a(p_a(w))$ for some $a \in [d]$.
Only such pairs can violate condition (C3).

Observe that any vertex $v$ has at most $d$ such grandparents.
We can construct, similarly as in the proof of Lemma~\ref{lem:star_coloring_implies_compatible_coloring}, an auxiliary directed graph of conflicts.
In-degree of any vertex in that graph is at most $d$ and the graph can be properly colored with $2d+1$ colors.

Finally, an out-coloring of $\dir{G}$ can be obtained as a product of coloring~$c$ and a coloring of the auxiliary graph.
The resulting coloring uses at most $(2d+1)p^m$ colors.
\end{proof}

\begin{proof}[Proof of Theorem~\ref{thm:main_theorem_existence}]
Let $\mathcal{F}$ be a class of graphs.
If the acyclic chromatic number of any graph in $\mathcal{F}$ is at most $r$, then $\mathcal{F}$ has a $k$\nobreakdash-universal graph on~$rk^{r-1}$ vertices as was shown by Alon and Marshall~\cite{AlonM98}.

For the proof of the other direction, suppose that $\mathcal{F}$ admits a $k$\nobreakdash-universal graph on $p$ vertices.
Let $d = \max(1,\ceil{ \log_k{p} })$.
Using Lemma~\ref{lem:lower_bounds_orientation_compatible} we get that any graph $G$ in $\mathcal{F}$ admits a $d$-orientation that has an out-coloring with $(2d+1)p^{\ceil{ \log_k{d} }}$ colors.
Such a coloring is an acyclic coloring of $G$.
It follows that any graph in $\mathcal{F}$ has the acyclic chromatic number bounded by a function of $p$ and $k$.
\end{proof}

We conjecture that the ceiling in our upper bound can be dropped and $\lambda_{\mathcal{F}}(k) = \Ot{k^{D(\mathcal{F})}}$.
We also believe that the multiplicative constant in Theorem~\ref{thm:main_theorem_upper_bound} can be improved.
Another question is whether our approach can be applied to mixed graphs, similarly to the results of Ne\v{s}et\v{r}il and Raspaud~\cite{NesetrilR98}.

\section{Acknowledgments}

We would like to thank Tomasz Krawczyk for the tremendous help with the preparation of this paper.
We would also like to thank Jaros\l{}aw Grytczuk, Jakub Kozik and Robert Obryk for many helpful comments.

\ifthenelse{\boolean{submission}}
{\section*{References}}
{}

\bibliographystyle{plain}
\bibliography{paper}

\end{document}